\documentclass[11pt]{article}
\usepackage{amssymb,float,amsthm,hyperref,xcolor,amsmath,tipa}
\hypersetup{
    colorlinks,
    citecolor=violet,
    filecolor=blue,
    linkcolor=blue,
    urlcolor=blue
}
\usepackage{thmtools}

\newtheorem{theorem}{Theorem}[section]
\newtheorem{lemma}[theorem]{Lemma}
\newtheorem{corollary}[theorem]{Corollary}
\newtheorem{proposition}[theorem]{Proposition}

\theoremstyle{definition}

\newtheorem{definition}[theorem]{Definition}
\newtheorem{problem}[theorem]{Question}

\newcommand{\cl}{\mathcal{L}}
\newcommand{\crr}{\mathcal{R}}

\usepackage[margin=1in]{geometry}
\counterwithin{figure}{section}

\title{Subsequence sums in permutations}
\author{ Collier Gaiser\thanks{Department of Mathematics, Community College of Aurora, Aurora, CO 80011,  United States of America. Email: {\tt colliergaiser@gmail.com}}
\and  Paul Horn\thanks{Department of Mathematics, University of Denver, Denver, CO 80208,  United States of America and Department of Mathematics and Applied Mathematics, University of Johannesburg, Johannesburg, South Africa. Email: {\tt paul.horn@du.edu}} 
}
\date{}
\begin{document}
\maketitle

\maketitle
\begin{abstract}
A sequence of positive integers $(a_1,a_2,\ldots,a_k)$ is called $\ell$-additive if $a_1+a_2+\cdots+a_k=\ell a_1$ or $\ell a_k$. In this paper, we prove that for all $k\geq3$, if $n$ is sufficiently large, then every permutation of $\{1,2,\ldots,n\}$ has a 2-additive subsequence of length $k$. We also provide polynomial bounds for the smallest $n$ such that every permutation of $\{1,2,\ldots,n\}$ has a 2-additive subsequence of length $k$. When only monotone subsequences are considered, we show that $18$ is the smallest $n$ such that every permutation of $\{1,2,\ldots,n\}$ has a monotone 2-additive subsequence of length three. Strong bounds are obtained for the minimum number of $\ell$-additive subsequences of any length, as well as monotone $2$-additive subsequences of length three. Using techniques in arithmetic Ramsey theory, we also show similar results for products and inverse sums. 
\end{abstract}


\maketitle

\section{Introduction}
In arithmetic Ramsey theory, the finitary version of the celebrated van der Waerden's theorem says that, for any positive integers $k$ and $r$, if $n$ is large enough, then every $r$-coloring of $\{1,2,\ldots,n\}$ has a monochromatic arithmetic progression of length $k$ (see \cite[Chapter 3]{GrahamButler2015} or \cite[Chapter 2]{LandmanRobertson2014}). In 1973, Entringer and Jackson \cite{EntringerJackson1973} asked whether similar results hold for permutations. For any positive integer $n$, let $S_n$ denote the set of permutations of $\{1,2,\ldots,n\}$ written in one-line notation. Entringer and Jackson asked whether every $p=(p_1,p_2,\ldots,p_n)\in S_n$ contains a monotone subsequence $(x,y,z)$ which is also a three-term arithmetic progression (3-AP), i.e., $x+z=2y$. Tom Odda (Ron Graham), R. C. Lyndon, and H. E. Thomas, Jr. \cite{EOLT1975} independently provided a negative answer to Entringer and Jackson's question. The main idea of Odda's proof is that one can build larger permutations without monotone 3-APs using smaller permutations without monotone 3-APs: if $(p_1,p_2,\ldots,p_n)\in S_n$ has no monotone 3-APs, then the permutation \[(2p_1,2p_2,\ldots,2p_n,2p_1-1,2p_2-1,\ldots,2p_n-1)\in S_{2n}\] also has no monotone 3-APs.

Although arithmetic progressions can be avoided by permutations, are there subsequences, with properties resembling results in arithmetic Ramsey theory, guaranteed to exist in every permutation of $\{1,2,\ldots,n\}$ for large enough $n$? Schur's theorem \cite[Chapter 8]{LandmanRobertson2014} and the more general Rado's theorem \cite[Chapter 9]{LandmanRobertson2014} say that, if $n$ is large enough, then monochromatic solutions to certain linear equations can always be found when $\{1,2,\ldots,n\}$ is colored with a fixed number of colors. In this paper, we prove some Schur-Rado-type results for permutations. Our first result is that, for a fixed $k$, if $n$ is large enough then every $p\in S_n$ contains a subsequence of length $k$ whose sum is either twice the first term or twice the last term.

\begin{theorem}\label{Theorem:Main}
    For any positive integer $k\geq3$ and large enough $n$, every $p\in S_n$ has a subsequence $(x_1,x_2,\ldots,x_k)$ such that $x_1+x_2+\cdots+x_k=2 x_1$ or $2x_k$.
\end{theorem}
In Theorem~\ref{Theorem:Main}, a permutation does not necessarily have both a subsequence $(x_1,x_2,\ldots,x_k)$ such that $x_1+x_2+\cdots+x_k=2 x_1$ and a subsequence $(y_1,y_2,\ldots,y_k)$ such that $y_1+y_2+\cdots+y_k=2 y_k$ simultaneously. For example, for any $n$, the identity permutation $(1,2,\ldots,n)\in S_n$ does not have a subsequence $(x_1,x_2,\ldots,x_k)$ such that $x_1+x_2+\cdots+x_k=2 x_1$. 

Our proof of Theorem~\ref{Theorem:Main} is purely combinatorial. The key observation is that any permutation of a subset of $\{1,2,\ldots,n\}$, which contains the integers $1,2,\ldots,2k-1$ and a selective collection of sums derived from these elements, is forced to contain a desirable subsequence. We also show that the smallest $n$ that satisfies Theorem~\ref{Theorem:Main} is between $(k-1)(3k-4)/2$ and $9(k-2)(k^2-k)^2/4$. 

We call a sequence $(a_1,a_2,\ldots,a_k)$ of positive integers \textbf{$2$-additive} if $a_1+a_2+\cdots+a_k=2a_1$ or $2a_k$. Theorem~\ref{Theorem:Main} then says that, for any positive integer $k\geq3$, if $n$ is large enough, then every $p\in S_n$ has a $2$-additive subsequence of length $k$. In general, we call a sequence $(a_1,a_2,\ldots,a_k)$  of positive integers $\ell$-\textbf{additive} if $a_1+a_2+\cdots+a_k=\ell a_1$ or $\ell a_k$. When $k=\ell$, for any $n$, the identity permutation $(1,2,\ldots,n)$ does not have $\ell$-additive subsequences of length $k$. When $k=3$ and $\ell=4$ and when $k=4$ and $\ell=3$, for any large enough $n$, every $p\in S_n$ has an $\ell$-additive subsequence of length $k$. See Section~\ref{Section:l-add} for details about these cases and the exact thresholds. The general case is still open.

\begin{problem}\label{Problem:l-Additive}
    Let $k,\ell\geq4$ be positive integers with $k\neq\ell$. Does there exist $n\in\mathbb{N}$ such that every $p\in S_n$ has an $\ell$-additive subsequence of length $k$?
\end{problem}
When every $p\in S_n$ has an $\ell$-additive subsequence of length $k$, using the prime number theorem, we prove that the minimum number of $\ell$-additive subsequences of length $k$ contained in any $p\in S_n$ is $\Omega(n/\log n)$ (treating $k$ and $\ell$ as constants). Using purely combinatorial arguments, we also show that the minimum number of $\ell$-additive subsequences of length $k$ contained in any $p\in S_n$ is $O(n^{k-1})$ (treating $k$ and $\ell$ as constants). In the previous two sentences, standard asymptotic notation was used: for functions $f(n)$ and $g(n)$, $f(n)=\Omega(g(n))$ if there exist constants $N$ and $c$ such that $|f(n)|\geq c|g(n)|$ for all $n\geq N$; and $f(n)=O(g(n))$ if there exist constants $N'$ and $C$ such that $|f(n)|\leq C|g(n)|$ for all $n\geq N'$.

The difference between additive subsequences and monotone 3-APs in permutations is not simply due to the fact that monotone 3-APs are required to be monotone. We show that Theorem~\ref{Theorem:Main} still holds for $k=3$ when the $2$-additive subsequence is required to be monotone.
\begin{theorem}\label{Theorem:Mon}
    If $n\geq 18$, then every $p\in S_n$ has a monotone $2$-additive subsequence of length three.
\end{theorem}
In Theorem~\ref{Theorem:Mon}, $18$ is the smallest $n$ possible because there exists a permutation of $\{1,2,\ldots,17\}$ which does not have monotone $2$-additive subsequences of length three. Whether similar results hold for monotone $2$-additive subsequences of length greater than three is an open question. 

\begin{problem}\label{Problem:Mon}
Is it true that, for any positive integer $k\geq4$, there exists $n$ such that every $p\in S_n$ contains a monotone $2$-additive subsequence of length $k$?
\end{problem}

Since the Erd\H{o}s–Szekeres theorem \cite{ErdosSzekeres1935} implies that every permutation of $\{1,2,\ldots,(k-1)^2+1\}$ has a monotone subsequence of length $k$, a positive answer to Question~\ref{Problem:Mon} would be a generalization of the Erd\H{o}s–Szekeres theorem.

Theorems~\ref{Theorem:Main} and \ref{Theorem:Mon} can be used to obtain similar results for subsequence products and inverse sums of subsequences. Using well-known arguments in arithmetic Ramsey theory, we show that for any $k$, if $n$ is large enough, then every $p\in S_n$ has a subsequence of length $k$ whose product is either the square of the first term or the square of the last term; and for any $k$, if $n$ is large enough, then every $p\in S_n$ has a subsequence $(x_1,x_2,\ldots,x_k)$ such that $1/x_1+1/x_2+\ldots+1/x_k=2/x_1$ or $2/x_k$. When the subsequences are required to be monotone, both results hold for the case $k=3$.

We note that arithmetic progressions behave differently in permutations of all positive integers. For any positive integer $k\geq3$, we call a monotone sequence $(x_1,x_2,\ldots,x_k)$ of length $k$ a monotone $k$-AP if there exists a positive integer $d$ such that $x_i=x_{i-1}+d$ for all $i\in\{2,3,\ldots,k\}$ or $x_{i}=x_{i+1}+d$ for all $i\in\{1,2,\ldots,k-1\}$. Davis, Entringer, Graham, and Simmons \cite{DEGS1977} showed that every permutation of the positive integers has a monotone $3$-AP, while there exist permutations of the positive integers without monotone 5-APs. However, despite some recent effort \cite{Adenwalla2024,Geneson2019,LeSaulnierVijay2011}, we still do not know whether there exists a permutation of the positive integers without monotone 4-APs. We also note that there are other studies on arithmetic progressions in permutations, such as arithmetic progressions in random permutations \cite{GZ2021}, arithmetic progressions in permutations of rational and real numbers \cite{ABJ2011}, and permutations such that both the terms and the indices avoid arithmetic progressions \cite{ES2017,Hegarty2004,HM2015,JS2015,SS2018}.

This paper is organized as follows. Some notation and preliminary results are presented in Section~\ref{Section:Pre}.  We prove Theorem~\ref{Theorem:Main} in Section~\ref{Section:2-Add} and provide polynomial bounds for the smallest $n$ such that every $p\in S_n$ has a $2$-additive subsequence of length $k$. In Section~\ref{Section:l-add}, we study general $\ell$-additive subsequences and provide results for two easiest cases not covered by Theorem~\ref{Theorem:Main} and, under the assumption that every $p\in S_n$ has an $\ell$-additive subsequence, bounds for the minimum number of $\ell$-additive subsequences. In Section~\ref{Section:Monotone2-Add}, using standard case analysis, we show that $18$ is the smallest $n$ such that every $p\in S_{n}$ has a monotone $2$-additive subsequence of length three. We also provide strong bounds for the minimum number of monotone $2$-additive subsequence of length three. In Section~\ref{Section:Prod}, we show similar results for subsequence products and inverse sums of subsequences. Finally, we discuss possible future directions in Section~\ref{Section:Conclu}.
\section{Preliminaries}\label{Section:Pre}

We first introduce some terminology which will be used in our proofs.

\begin{definition}
    For all $p=(p_1,p_2,\ldots,p_n)\in S_n$ and $s\in\{1,2,\ldots,n\}$, if $p_i=s$, then we define $\mathcal{L}_p(s)=\{p_j:j<i\}$ and $\mathcal{R}_p(s)=\{p_j:j>i\}$.
\end{definition}
That is, $\mathcal{L}_p(s)$ consists of all the terms to the left of $s$ and $\mathcal{R}_p(s)$ consists of all the terms to the right of $s$ when we arrange the terms of $p$ on a horizontal line. For example, for the permutation $p=(5,1,4,3,2)\in S_5$, we have $\mathcal{L}_p(4)=\{1,5\}$ and $\mathcal{R}_p(1)=\{2,3,4\}$.

\begin{definition}
    Let $n\in\mathbb{N}$, $A\subseteq\{1,2,\ldots,n\}$, and $p\in S_n$. The subpermutation of $p$ on $A$ is a sequence obtained by deleting all the terms of $p$ which are not in $A$, but keeping the relative order of the terms that are in $A$.
\end{definition}
For example, $(5,1,3)$ is a subpermutation of $(5,1,4,3,2)$ on $\{1,3,5\}$. If $q$ is a subpermutation of $p$ on $A$, then we will simply call $q$ a subpermutation on $A$ when there is no confusion.

Now we state two simple observations which will be used to simplify our proofs.
\begin{lemma}\label{Lemma:Fromltol-1}
    A sequence $(x_1,x_2,\ldots,x_k)$ is $\ell$-additive if and only if $\sum_{i=1}^{k-1}x_i=(\ell-1)x_k$ or $\sum_{i=2}^kx_i=(\ell-1)x_1$.
\end{lemma}
\begin{lemma}\label{Lemma:FromAll-nToSome-n}
Let $m,k,\ell\in\mathbb{N}$. If every $p\in S_m$ has an $\ell$-additive subsequence of length $k$, then for all $n\geq m$, every $p\in S_n$ has an $\ell$-additive subsequence of length $k$. 
\end{lemma}

Lemmas~\ref{Lemma:Fromltol-1} and \ref{Lemma:FromAll-nToSome-n} also hold for monotone $\ell$-additive subsequences.
\section{2-Additive Subsequences of Arbitrary Length}\label{Section:2-Add}
In this section, we first prove Theorem~\ref{Theorem:Main}. We do so by focusing on a subset of $\{1,2,\ldots,n\}$ which contains $1,2,\ldots,2k-1$ and a selective collection of sums derived from these elements. Because of this, our proof also provides an upper bound for the smallest $n$ such that every $p\in S_n$ has a 2-additive subsequence of length $k$. 

We start with some notation. Given $k\geq3$, for all $p\in S_{2k-1}$, let \[\alpha_p=p_1+p_2+\cdots+p_{k-1},\] \[\beta_p=p_{k+1}+p_{k+2}+\cdots+p_{2k-1},\]
    \[
    U_p=\{(k-2)\alpha_p,\alpha_p-p_1,\alpha_p-p_2,\ldots,\alpha_p-p_{k-1}\},
    \]
    \[
    V_p=\{(k-2)\beta_p,\beta_p-p_{k+1},\beta_p-p_{k+2},\ldots,\beta_p-p_{2k-1}\}.
    \]
    and \[L_p=\max\{\text{lcm}(a,b):a\in U_p,b\in V_p\},\] where $\text{lcm}(a,b)$ is the least common multiple of $a$ and $b$. Notice that, for all $p\in S_{2k-1}$, there are exactly $k$ elements in the set $U_p$ as well as in the set $V_p$. Let $N_k=\max_{p\in S_{2k-1}}L_p$. Our goal is to show that every $\sigma\in S_{N_k}$ has a $2$-additive susbequence of length $k$.

Now we prove an upper bound for $N_k$.

\begin{lemma}\label{Lemma:UppderBoundN}
For all $k\geq3$, we have $N_k\leq 9(k-2)(k^2-k)^2/4$.
\end{lemma}
\begin{proof}
Let $k\geq3$, $p\in S_{2k-1}$, $a\in U_p$, and $b\in V_p$. By the definition of $U_p$ and $V_p$, we have
\[
\text{lcm}(a,b)\leq(k-2)\alpha_p\beta_p<(k-2)\left(\sum_{i=k+1}^{2k-1}i\right)^2=\frac{9}{4}(k-2)(k^2-k)^2.
\]
Hence, we have $L_p\leq 9(k-2)(k^2-k)^2/4$ for all $p\in S_{2k-1}$. Therefore, $N_k\leq 9(k-2)(k^2-k)^2/4$.
\end{proof}
    
Next, we show a sufficient condition for a permutation of $\{1,2,\ldots,N_k\}$ to have a $2$-additive subsequence of length $k$.    
\begin{lemma}\label{Lemma:NonEmptyIntersection}
    Let $k\geq3$, $\sigma\in S_{N_k}$, and $p$ the subpermutation of $\sigma$ on $\{1,2,\ldots,2k-1\}$. If $U_p\cap\mathcal{L}_\sigma(p_k)=\emptyset$ or $V_p\cap\mathcal{R}_\sigma(p_k)=\emptyset$, then $\sigma$ has a $2$-additive subsequence of length $k$.
\end{lemma}
\begin{proof}
    Let $k\geq3$, $\sigma\in S_{N_k}$, and $p$ the subpermutation of $\sigma$ on $\{1,2,\ldots,2k-1\}$. By symmetry, it suffices to prove that if $U_p\cap\mathcal{L}_\sigma(p_k)=\emptyset$, then $\sigma$ has a $2$-additive subsequence of length $k$. We assume that $U_p\cap\mathcal{L}_\sigma(p_k)=\emptyset$. Since $\{p_1,p_2,\ldots,p_{k-1}\}\subseteq\mathcal{L}_\sigma(p_k)$ and $(k-2)\alpha_p\in U_p$, we have $\{p_1,p_2,\ldots,p_{k-1}\}\subseteq\mathcal{L}_\sigma((k-2)\alpha_p)$. There are two cases depending on the locations of $\alpha_p-p_1,\alpha_p-p_2,\ldots,\alpha_p-p_{k-1}$ relative to $(k-2)\alpha_p$.

    {\bf Case 1}: $\alpha_p-p_i\in\mathcal{R}_\sigma((k-2)\alpha_p)$ for all $i\in\{1,2,\ldots,k-1\}$. Since $(\alpha_p-p_1)+(\alpha_p-p_2)+\cdots+(\alpha_p-p_{k-1})=(k-2)\alpha_p$, $\sigma$ has a $2$-additive subsequence of length $k$ whose sum is $2(k-2)\alpha_p$.

    {\bf Case 2}: There exists $i\in\{1,2,\ldots,k-1\}$ such that $\alpha_p-p_i\in\mathcal{L}_\sigma((k-2)\alpha_p)$. WLOG, we assume that $\alpha_p-p_1\in\mathcal{L}_\sigma((k-2)\alpha_p)$. Write $A_1=\alpha_p-p_1$ and, for all $i\in\{2,3,\ldots,k-1\}$, $A_i=A_{i-1}+\alpha_p-p_i$. Let $j\in\{2,3,\ldots,k-1\}$ be the smallest index such that $A_j\notin\mathcal{L}_\sigma((k-2)\alpha_p)$. Notice that since $A_{k-1}=(\alpha_p-p_1)+(\alpha_p-p_2)+\cdots+(\alpha_p-p_{k-1})=(k-2)\alpha_p\notin\mathcal{L}_p((k-2)\alpha_p)$, such a $j$ exists. Now we have $\{p_1,p_2,\ldots,p_{k-1},A_{j-1}\}\subseteq\mathcal{L}_\sigma((k-2)\alpha_p)$, but $A_j\notin\mathcal{L}_\sigma((k-2)\alpha_p)$. Since $A_j=A_{j-1}+p_1+p_2+\cdots+p_{j-1}+p_{j+1}+p_{j+2}+\cdots+p_{k-1}$, $\sigma$ has a $2$-additive subsequence of length $k$ whose sum is $2A_j$.
\end{proof}

Now we are ready to prove Theorem~\ref{Theorem:Main}. Here we provide a more detailed statement which contains an upper bound for $n$ in Theorem~\ref{Theorem:Main}.
\begin{theorem}\label{Theorem:Existence:Nonmonotone}
For all $k\geq3$ and $n\geq9(k-2)(k^2-k)^2/4$, every $p\in S_n$ has a $2$-additive subsequence of length $k$.
\end{theorem}

\begin{proof}
We will show that every $\sigma\in S_{N_k}$ has a 2-additive subsequence of length $k$. Suppose, by way of contradiction, that there exists $\sigma\in S_{N_k}$ which does not have 2-additive subsequences of length $k$. Let $p$ be the subpermutation of $\sigma$ on $\{1,2,\ldots,2k-1\}$.

By Lemma~\ref{Lemma:NonEmptyIntersection}, $U_p\cap\mathcal{L}_\sigma(p_k)\neq\emptyset$ and $V_p\cap\mathcal{R}_\sigma(p_k)\neq\emptyset$. Let $a\in U_p\cap \cl_\sigma(p_{k})$ and $b\in V_p\cap \crr_\sigma(p_{k})$. We will show that if $h$ is a positive integer and $ha\leq N_k$, then $h a\in \cl_\sigma(p_{k})$. Then by symmetry, if $h$ is positive integer and $hb\leq N_k$, then $h b\in \mathcal{R}_\sigma(p_{k})$. Since $a\in U_p$, either $a=(k-2)\alpha_p$ or $a=\alpha-p_i$ for some $i\in\{1,2,\ldots,k-1\}$.

\textbf{Case 1}: $a=(k-2)\alpha_p$. Let $S_0=(k-2)\alpha_p$ and, for all $j\in\{1,2,\ldots,k-1\}$, let $S_j=S_{j-1}+p_1+p_2+\cdots+p_{j-1}+p_{j+1}+\cdots+p_{k-1}$. We will show that $S_j\in \cl_\sigma(p_{k})$ for all $j=\{1,2,\ldots,k-1\}$. Suppose not. Let $j\in\{1,2,\ldots,k-1\}$ be the smallest such that $S_j\notin \cl_\sigma(p_{k})$. Then we have $S_j=S_{j-1}+p_1+p_2+\cdots+p_{j-1}+p_{j+1}+\cdots+p_{k-1}$ and $\{S_{j-1},p_1,p_2,\ldots,p_{k-1}\}\subseteq\cl_\sigma(S_j)$ which is a contradiction. By our construction $S_{k-1}=2(k-2)\alpha_p$ and hence $2(k-2)\alpha_p\in\cl_\sigma(p_{k})$. By induction, if $h\geq1$ and $ha\leq N_k$, then $ha\in\mathcal{L}_\sigma(p_k)$.

\textbf{Case 2}: $a=\alpha_p-p_i$ for some $i\in\{1,2,\ldots,k-1\}$. Since $2a=a+p_1+p_2+\cdots+p_{i-1}+p_{i+1}+p_{i+2}+\cdots+p_{k-1}$, $\{a,p_1,p_2,\ldots,p_{k-1}\}\subseteq\cl_\sigma(p_k)$, and $\sigma$ does not have $2$-additive subsequences of length $k$, we have $2a\in\mathcal{L}_\sigma(p_k)$. By induction, if $h\geq1$ and $ha\leq N_k$, then $ha\in\mathcal{L}_\sigma(p_k)$.

Since $a\in U_p$ and $b\in V_p$, there exist positive integers $a'$ and $b'$ such that $a'a=b'b\leq L_p$. Since $L_p\leq N_k$, we have $b'b=a'a\in \cl_\sigma(p_{k})$ and $a'a=b'b\in \crr_\sigma(p_{k})$ which is a contradiction.

Hence, every $p\in S_{N_k}$ has a $2$-additive subsequence of length $k$. By Lemma~\ref{Lemma:UppderBoundN}, $N_k\leq 9(k-2)(k^2-k)^2/4$. Therefore, for all $n\geq 9(k-2)(k^2-k)^2/4$, every $p\in S_n$ has a $2$-additive subsequence of length $k$.
\end{proof}

For all $k\geq3$, let $f(k,2)$ be the smallest $n$ such that every $p\in S_n$ has a 2-additive subsequence of length $k$. By Theorem~\ref{Theorem:Existence:Nonmonotone}, we have $f(k,2)\leq9(k-2)(k^2-k)^2/4$. Now we prove a lower bound for $f(k,2)$.

\begin{theorem}\label{Theorem:NonmonotoneLower}
For all $k\geq3$, $f(k,2)\geq(k-1)(3k-4)/2$.
\end{theorem}
\begin{proof}
Write $m=(k-1)(3k-4)/2$. Consider the permutation \[p=(1,2,\ldots,k-2,m-1,m-2,\ldots,k-1)\in S_{m-1}.\] 
To prove that $f(k,2)\geq(k-1)(3k-4)/2$, it suffices to show that $p$ does not have $2$-additive subsequences of length $k$. Notice that, by Lemma~\ref{Lemma:Fromltol-1}, $p$ does not have $2$-additive subsequences of length $k$ if and only if, for all $s\in\{1,2,\ldots,m-1\}$, neither $\cl_p(s)$ nor $\crr_p(s)$ contains a subset of size $k-1$ which sums to $s$. This is obviously true if $s<(k-1)k/2$. So we suppose that $s\geq(k-1)k/2$. By the construction of $p$, $\cl_p(s)$ only contains $k-2$ numbers that are smaller than $s$. So $\cl_p(s)$ does not contain a subset of size $k-1$ which sums to $s$. As for $\crr_\sigma(s)$, the sum of a subset of size $k-1$ is at least
\[
\sum_{i=k-1}^{2k-3}i=\frac{1}{2}(k-1)(3k-4)>m-1.
\]
Hence, $\crr_\sigma(s)$ does not contain a subset of size $k-1$ which sums to $s$. Therefore, $p$ does not have a $2$-additive subsequence of length $k$.
\end{proof}
The lower bound in Theorem~\ref{Theorem:NonmonotoneLower} matches the exact value for $f(k,2)$ when $k=3$. 

\begin{proposition}\label{Prop:Extremal2Add}
$5$ is the smallest $n$ such that every $p\in S_n$ has a $2$-additive subsequence of length three; that is $f(3,2)=5$.
\end{proposition}
\begin{proof}
    By Theorem~\ref{Theorem:NonmonotoneLower}, we have $f(3,2)\geq(3-1)(3\cdot3-4)/2=5$. To show that $f(3,2)=5$, we only need to show that every $p\in S_5$ has a $2$-additive subsequence of length three. One way to do this is to check all the permutations of $\{1,2,3,4,5\}$. However, this becomes impractical for longer permutations and hence we provide a proof using proof-by-contradiction and elementary case analysis. This approach will also be used in Sections~\ref{Section:l-add} and \ref{Section:Monotone2-Add} for longer permutations.
    
    Suppose, by way of contradiction, that $p\in S_5$ does not have $2$-additive subsequences of length three. By symmetry, we assume that the subpermutation of $p$ on $\{1,2\}$ is $(1,2)$. Since $1+2=3$, we have $3\in\mathcal{L}_p(2)\cap\mathcal{R}_p(1)$. Hence, the subpermutation of $p$ on $\{1,2,3\}$ is $(1,3,2)$. Since $1+3=4$ and $3+2=5$, we must have $p=(1,4,3,5,2)$. This is a contradiction because $(1,4,5)$ is a $2$-additive subsequence of $p$. Hence every $p\in S_5$ has a $2$-additive subsequence of length three. Therefore, $f(3,2)=5$.
\end{proof}
\section{$\ell$-Additive Subsequences}\label{Section:l-add}
We study $\ell$-additive subsequences in this section. We first note that if $k=\ell$, then we can find permutations without $\ell$-additive subsequences of length $k$.

\begin{proposition}
    If $k=\ell$, then, for all $n$, $(1,2,\ldots,n)\in S_n$ does not have an $\ell$-additive subsequence of length $k$.
\end{proposition}
\begin{proof}
    Let $(x_1,x_2,\ldots,x_k)$ be a subsequence of $(1,2,\ldots,n)$. Then we have $x_1+x_2+\cdots+x_k>kx_1=\ell x_1$ and $x_1+x_2+\cdots+x_k<kx_k=\ell x_k$. Hence, $(1,2,\ldots,n)$ does not have an $\ell$-additive subsequence.
\end{proof}

We don't know the answer to the general case when $k\neq\ell$. However, we have affirmative answers to two cases. For the first case, we have $k=4$ and $\ell=3$; and for the second case, we have $k=3$ and $\ell=4$. For positive integers $k$ and $\ell$, let $f(k,\ell)$ be the smallest $n$, if it exists, such that every $p\in S_n$ has an $\ell$-additive subsequence of length $k$.

\begin{proposition}
$8$ is the smallest $n$ such that every $p\in S_n$ has a $3$-additive subsequence of length four; that is, $f(4,3)=8$.
\end{proposition}

\begin{proof}
The permutation $(2,4,6,7,5,3,1)\in S_7$ does not have $3$-additive subsequences of length four. Hence $f(4,3)\geq8$. To show that $f(4,3)=8$, it remains to show that every $p\in S_8$ has a $3$-additive subsequence of length four. By Lemma~\ref{Lemma:Fromltol-1}, it suffices to show that every $p\in S_8$ has a subsequence $(x_1,x_2,x_3,x_4)$ such that either $x_1+x_2+x_3=2x_4$ or $x_2+x_3+x_4=2x_1$.  Suppose, by way of contradiction, that $p\in S_8$ is a permutation that does not have this property. WLOG, we assume that the subpermutation on $\{1,2\}$ is $(1,2)$. Then the subpermutation of $p$ on $\{1,2,5\}$ is $(5,1,2)$, $(1,2,5)$, or $(1,5,2)$.

\textbf{Case 1}: The subpermutation on $\{1,2,5\}$ is $(5,1,2)$. Since $1+2+5=2\cdot4$, we have $4\in\mathcal{R}_p(5)\cap\mathcal{L}_p(2)$. Since $1+2+7=2\cdot5$ and $1,2\in\mathcal{R}_p(5)$, we have $7\in\mathcal{L}_p(5)$. Since $1+5+8=2\cdot7$ and $1,5\in\mathcal{R}_p(7)$, we have $8\in\mathcal{L}_p(7)$. Hence $(8,7,5,4)$ is a subpermutation of $p$ with $4+5+7=2\cdot8$. This is a contradiction.

\textbf{Case 2}: The subpermutation on $\{1,2,5\}$ is $(1,2,5)$. Since $1+2+5=2\cdot4$, we have $4\in\mathcal{R}_p(1)\cap\mathcal{L}_p(5)$. Since $1+2+7=2\cdot5$ and $1,2\in\mathcal{L}_p(5)$, we have $7\in\mathcal{R}_p(5)$. Since $1+5+8=2\cdot7$ and $1,5\in\mathcal{L}_p(7)$, we have $8\in\mathcal{R}_p(7)$. Hence $(4,5,7,8)$ is a subpermutation of $p$ with $4+5+7=2\cdot8$. This is a contradiction.

\textbf{Case 3}: The subpermutation on $\{1,2,5\}$ is $(1,5,2)$. Since $1+2+5=2\cdot4$, we have $4\in\mathcal{R}_p(1)\cap\mathcal{L}_p(2)$. So the subpermutation on $\{1,2,4,5\}$ is either $(1,4,5,2)$ or $(1,5,4,2)$. Now we split Case 3 into six disjoint subcases based on the subpermutation on $\{1,2,4,5\}$ and the location of $8$.

\textit{Subcase 3.1}: The subpermutation on $\{1,2,4,5\}$ is $(1,4,5,2)$ and $8\in\mathcal{R}_p(5)$. Since $1+5+8=2\cdot7$ and $\{1,5\}\subseteq\mathcal{L}_p(8)$, we must have $7\in\mathcal{L}_p(8)$. Now the subpermutation on $\{4,5,7,8\}$ is $(4,5,7,8)$, $(4,7,5,8)$, or $(7,4,5,8)$. Since $4+5+7=2\cdot8$, we have a contradiction.

\textit{Subcase 3.2}: The subpermutation on $\{1,2,4,5\}$ is $(1,4,5,2)$ and $8\in\mathcal{L}_p(4)$. Since $2+4+8=2\cdot7$ and $\{2,4\}\subseteq\mathcal{R}_p(8)$, we have $7\in\mathcal{R}_p(8)$. Now the subpermutation on $\{4,5,7,8\}$ is $(8,7,4,5)$, $(8,4,7,5)$, or $(8,4,5,7)$. Since $4+5+7=2\cdot8$, we have a contradiction.

\textit{Subcase 3.3}: The subpermutation on $\{1,2,4,5,8\}$ is $(1,4,8,5,2)$. Since $1+5+8=2\cdot7$ and $\{1,8\}\subseteq\mathcal{L}_p(5)$, we have $7\in\mathcal{L}_p(5)$. Since $2+4+8=2\cdot7$ and $\{2,8\}\subseteq\mathcal{R}_p(4)$, we have $7\in\mathcal{R}_p(4)$. Since $1+4+7=2\cdot6$, we have $6\in\mathcal{L}_p(7)\subseteq\mathcal{L}_p(5)$. Since $1+3+6=2\cdot5$ and $\{1,6\}\subseteq\mathcal{L}_p(5)$, we have $3\in\mathcal{R}_p(5)$. Now we have $\{2,3,7\}\subseteq\mathcal{R}_p(6)$ and $2+3+7=2\cdot6$ which is a contradiction.

\textit{Subcase 3.4}: The subpermutation on $\{1,2,4,5\}$ is $(1,5,4,2)$ and $8\in\mathcal{L}_p(5)$. Since $1+5+8=2\cdot7$ and $\{1,8\}\subseteq\mathcal{L}_p(5)$, we have $7\in\mathcal{L}_p(5)$. Since $2+4+8=2\cdot7$ and $2,4\in\mathcal{R}_p(8)$, we have $7\in\mathcal{R}_p(8)$. Now $(8,7,5,4)$ is a subpermutation of $p$ with $4+5+7=2\cdot8$. This is a contradiction.

\textit{Subcase 3.5}: The subpermutation on $\{1,2,4,5\}$ is $(1,5,4,2)$ and $8\in\mathcal{R}_p(4)$. Since $2+4+8=2\cdot7$ and $\{2,8\}\subseteq\mathcal{R}_p(4)$, we have $7\in\mathcal{R}_p(4)$. Since $1+5+8=2\cdot7$ and $\{1,5\}\subseteq\mathcal{L}_p(8)$, we have $7\in\mathcal{L}_p(8)$. Now $(5,4,7,8)$ is a subpermutation of $p$ with $4+5+7=2\cdot8$. This is a contradiction.

\textit{Subcase 3.6}: The subpermutation on $\{1,2,4,5,8\}$ is $(1,5,8,4,2)$. Since $2+4+8=2\cdot7$, we have $7\in\mathcal{R}_p(8)$. Now $(1,5,8,7)$ is a subpermutation of $p$ with $1+5+8=2\cdot7$. This is a contradiction.
\end{proof}    

\begin{proposition}
$13$ is the smallest $n$ such that very $p\in S_n$ has a $4$-additive subsequence of length three; that is, $f(3,4)=13$.
\end{proposition}

\begin{proof}
$(9,10,7,1,4,3,2,5,6,11,12,8)\in S_{12}$ does not have $4$-additive subsequences of length three. Hence, we have $f(3,4)\geq13$. To show that $f(3,4)=13$, it remains to show that every $p\in S_{13}$ has a $4$-additive subsequence of length three. By Lemma~\ref{Lemma:Fromltol-1}, it suffices to show that every $p\in S_{13}$ has a subsequence $(x_1,x_2,x_3)$ such that either $x_1+x_2=3x_3$ or $x_2+x_3=3x_1$. Suppose, for a contradiction, that $p\in S_{13}$ is a permutation that does not have this property. WLOG, we assume that the subpermutation on $\{1,5\}$ is $(1,5)$. 

    Since $1+5=3\cdot2$, the subpermutation on $\{1,2,5\}$ is $(1,2,5)$. Since $2+13=3\cdot5$, we have $13\in\mathcal{R}_p(5)$. Since $5+13=3\cdot6$, $6\in\mathcal{R}_p(5)\cap \mathcal{L}_p(13)$. Now the subpermutation on $\{1,2,5,6,13\}$ is $(1,2,5,6,13)$.

    Since $7+8=3\cdot5$, the subpermutation of $p$ on $\{5,7,8\}$ is either $(7,5,8)$ or $(8,5,7)$. Since $8+13=3\cdot7$, the subpermuation on $(7,8,13)$ is either $(8,7,13)$ or $(13,7,8)$. Since $13\in\mathcal{R}_p(5)$, $(13,7,8)$ is impossible. Hence $7\in\mathcal{R}_p(5)$ and $8\in\mathcal{L}_p(5)$.

    Since $5+7=3\cdot4$, $4\in\mathcal{R}_p(5)$. Since $5+4=3\cdot3$, $3\in\mathcal{R}_p(5)$. Now either $(1,8,3)$ or $(8,1,3)$ is a subsequence of $p$ and $1+8=3\cdot3$ which is a contradiction.
\end{proof}    
Now we look at the number of $\ell$-additive subsequences of length $k$ in a permutation. Given $k$, $\ell$, $n$, and $p\in S_n$, let $F_p(k,\ell,n)$ be the number of $\ell$-additive subsequences of length $k$ in $p$. Then let 
\[
F(k,\ell,n)=\min_{p\in S_n}F_p(k,\ell,n).
\]

\begin{theorem}\label{Theorem:NumberOfAdditiveGeneral}
For all $k$, $\ell$, and $n$, we have
    \[
    F(k,\ell,n)\leq\frac{2}{k}\left(\frac{e}{k-1}\right)^{k-1}n^{k-1};
    \]
    and given $k$ and $\ell$, if $f(k,\ell)$ exists, then for all large enough $n$, we have
    \[
    F(k,\ell,n)\geq \frac{n}{f(k,\ell)[\log n-\log f(k,\ell)+2]}-\frac{f(k,\ell)}{\log f(k,\ell)-4}+1.
    \]
\end{theorem}
\begin{proof}
We first prove the upper bound. The number of permutations containing a fixed $\ell$-additive subsequence of length $k$ is $(n-k)!{n \choose k}$. There are ${n\choose k-1}$ ways of choosing $k-1$ distinct numbers in $\{1,2,\ldots,n\}$; and for each $k-1$ distinct numbers, there are $(k-1)!$ ways to arrange them and \textit{at most} two ways to add another number to make it an $\ell$-additive sequence of length $k$. Hence, the total number of $\ell$-additive subsequences of length $k$ in all permutations of $\{1,2,\ldots,n\}$ is \textit{at most} 
    \[
    2{n\choose k-1}(k-1)!(n-k)!{n \choose k}=\frac{2}{k}{n\choose k-1}n!.
    \]
Since there are $n!$ permutations in $S_{n}$, there exists some $p\in S_n$ which has \textit{at most} \[\frac{2}{k}{n\choose k-1}\frac{n!}{n!}=\frac{2}{k}{n\choose k-1}\leq\frac{2}{k}\left(\frac{en}{k-1}\right)^{k-1}=\frac{2}{k}\left(\frac{e}{k-1}\right)^{k-1}n^{k-1}\]
$\ell$-additive subsequences of length $k$. This completes the proof of the upper bound.

Now we prove the lower bound. Consider a large enough $n$ and $p\in S_n$. Write $b=f(k,\ell)$. Let $a$ be a prime such that $a>b$ and $ab\leq n$. Since every permutation of $\{1,2,\ldots,b\}$ has an $\ell$-additive subsequence of length $k$, every subpermutation of $p$ on $\{a,2a,\ldots,ba\}$ has an $\ell$-additive subsequence of length $k$. A result by Rosser \cite{Rosser1941} on the Prime Number Theorem \cite[p.~274]{Nathanson2000} implies that, there are at least $(n/b)/(\log(n/b)+2)$ primes $a$ in $\{1,2,\ldots,n\}$ such that $ab\leq n$  and at most $b/(\log b-4)$ primes in $\{1,2,\ldots,b\}$. Hence, the number of primes $a$ such that $a>b$ and $ab\leq n$ is at least
\[
\frac{n/b}{\log(n/b)+2}-\frac{b}{\log b-4}=\frac{n}{b(\log n-\log b+2)}-\frac{b}{\log b-4}.
\]
For each of these primes $a$, the subpermutation of $p$ on $\{a,2a,\ldots,ba\}$ has an $\ell$-additive subsequence of length $k$. Since every term of each $\ell$-additive subsequence obtained this way has a unique prime factor greater than $b$, these $\ell$-additive subsequences are all distinct. Adding an $\ell$-additive subsequence in the subpermutation of $p$ on $\{1,2,\ldots,b\}$, $p$ has \textit{at least} $n/[b(\log n-\log b+2)]-b/(\log b-4)+1$ $\ell$-additive subsequences of length $k$. This completes the proof of the lower bound.
\end{proof}

\section{Monotone 2-Additive Subsequences of Length Three}\label{Section:Monotone2-Add}

For all positive integers $k$ and $\ell$, let $g(k,\ell)$ be the smallest $n$ such that every $p\in S_n$ has a monotone $\ell$-additive subsequence of length $k$, if it exists. In this section, we first show that $g(3,2)=18$. Our proof is based on elementary combinatorial arguments. To ease the description, we divide the proof into four lemmas and then combine them together.

\begin{lemma}\label{Lemma:Monotone132}
Let $p\in S_{12}$. If the subpermutation of $p$ on $\{1,2,3\}$ is $(1,3,2)$, then $p$ has a monotone $2$-additive subsequence of length three.
\end{lemma}
\begin{proof}
    Let $p\in S_{12}$ be such that the subpermutation of $p$ on $\{1,2,3\}$ is $(1,3,2)$. Suppose, by way of contradiction, that $p$ does not have a monotone $2$-additive subsequence of length three. Since $1+3=4$, we have $4\in\mathcal{L}_p(3)$ because otherwise $(1,3,4)$ would be a monotone $2$-additive subsequence of length three. Similarly, since $3+2$, we have $5\in\mathcal{R}_p(3)$. Now we have $\{1,4\}\subseteq\mathcal{L}_p(5)$. Since $1+4=5$, we must have $4\in\mathcal{L}_p(1)$. To summarize, the subpermutation of $p$ on $\{1,2,3,4,5\}$ is either $(4,1,3,5,2)$ or $(4,1,3,2,5)$.

    {\bf Case 1}: The subpermutation of $p$ on $\{1,2,3,4,5\}$ is $(4,1,3,5,2)$. Since $2+5=7$, we have $7\in\mathcal{R}_p(5)$. Since $2+4=6$ and $1+5=6$, $6\in\mathcal{L}_p(5)\cap\mathcal{R}_p(4)$. Now we have $\{1,6\}\subseteq\mathcal{L}_p(7)$. Since $1+6=7$, we must have $6\in\mathcal{L}_p(1)$. Since $4+6=10$, we have $10\in\mathcal{L}_p(6)$. Since $3+5=8$ and $2+6=8$, we have $8\in\mathcal{L}_p(5)\cap\mathcal{R}_p(6)$. Now the subpermutation on $\{2,8,10\}$ is $(10,8,2)$ which is a contradiction.

    {\bf Case 2}: The subpermutation of $p$ on $\{1,2,3,4,5\}$ is $(4,1,3,2,5)$. Since $4+2=6$ and $1+5=6$, we have $6\in\mathcal{L}_p(5)\cap\mathcal{R}_p(4)$. Since $4+3=7$ and $2+5=7$, we have $7\in\mathcal{L}_p(5)\cap\mathcal{R}_p(4)$. Now the subpermutation of $p$ on $\{4,5,6,7\}$ is either $(4,6,7,5)$ or $(4,7,6,5)$. If it is $(4,7,6,5)$, then since $4+7=11$, we have $11\in\mathcal{L}_p(7)$ and hence $(11,6,5)$ is a subpermutation of $p$ which is a contradiction. Hence the subpermutation of $p$ on $\{4,5,6,7\}$ is $(4,6,7,5)$. Since $4+7=11$ and $6+5=11$, we have $11\in\mathcal{L}_p(7)\cap\mathcal{R}_p(6)$. Since $5+7=12$, we have $12\in\mathcal{R}_p(7)$. Since $1+11=12$ and $11\in\mathcal{L}_p(12)$, we must have $11\in\mathcal{L}_p(1)$. Since $2+6=8$, we have $8\in\mathcal{R}_p(6)$. Now $\{8,12\}\subseteq\mathcal{R}_p(4)$. Since $4+8=12$, we have $8\in\mathcal{R}_p(12)$. Since $4+6=10$, we have $10\in\mathcal{L}_p(6)$. Now we have $\{3,7\}\subseteq\mathcal{R}_p(10)$. Since $3+7=10$, we have $7\in\mathcal{R}_p(3)$. Now $(1,7,8)$ is a monotone $2$-additive subsequence of $p$ which is a contradiction.
\end{proof}

\begin{lemma}\label{Lemma:Monotone3412}
Let $p\in S_{16}$. If the subpermutation of $p$ on $\{1,2,3,4\}$ is $(3,4,1,2)$, then $p$ has a monotone $2$-additive subsequence of length three.
\end{lemma}
\begin{proof}
    Let $p\in S_{16}$ be such that the subpermutation of $p$ on $\{1,2,3,4\}$ is $(3,4,1,2)$. Suppose, by way of contradiction, that $p$ does not have a monotone $2$-additive subsequence of length three. Since $3+4=7$, we have $7\in\mathcal{L}_p(4)$. Since $4+1=5$, we have $5\in\mathcal{R}_p(4)$. Now $\{2,5\}\subseteq\mathcal{R}_p(7)$. Since $5+2=7$, we have $5\in\mathcal{R}_p(2)$. Notice that now the subpermutation of $p$ on $\{1,2,3,4,5\}$ is $(3,4,1,2,5)$ and we also have $7\in\mathcal{L}_p(4)$. Since $4+2=6$ and $5+1=6$, we have $6\in\mathcal{L}_p(5)\cap\mathcal{R}_p(4)$. Since $1+6=7$ and $\{1,6\}\subseteq\mathcal{R}_p(7)$, we have $6\in\mathcal{R}_p(1)$. Hence, the subpermutation of $p$ on $\{1,2,3,4,5,6\}$ is either $(3,4,1,6,2,5)$ or $(3,4,1,2,6,5)$.

    {\bf Case 1}: The subpermutation of $p$ on $\{1,2,3,4,5,6\}$ is $(3,4,1,6,2,5)$. Since $6+2=8$ and $3+5=8$, we have $8\in\mathcal{L}_p(5)\cap\mathcal{R}_p(6)$. Since $3+6=9$ and $7+2=9$, we have $9\in\mathcal{L}_p(6)\cap\mathcal{R}_p(7)$. Since $8+5=13$, we have $13\in\mathcal{R}_p(8)$. Now since $\{4,9\}\subseteq\mathcal{L}_p(13)$ and $4+9=13$, we have $9\in\mathcal{L}_p(4)$. Notice that the subpermutation of $p$ on $\{1,4,6,7,9\}$ is now $(7,9,4,1,6)$. Since $4+6=10$ and $9+1=10$, we have $10\in\mathcal{L}_p(6)\cap\mathcal{R}_p(9)$. Since $7+9=16$, we have $16\in\mathcal{L}_p(9)$. Now $(16,10,6)$ is a monotone $2$-additive subsequence of $p$ which is a contradiction.

    {\bf Case 2}: The subpermutation of $p$ on $\{1,2,3,4,5,6\}$ is $(3,4,1,2,6,5)$. Since $6+5=11$, we have $11\in\mathcal{R}_p(6)$. Since $7+1=8$ and $6+2=8$, we have $8\in\mathcal{L}_p(6)\cap\mathcal{R}_p(7)$. Now $\{3,8\}\subseteq\mathcal{L}_p(11)$. Since $3+8=11$, we have $8\in\mathcal{L}_p(3)$. Notice that now the subpermutation of $p$ on $\{1,2,3,4,5,6,7,8\}$ is $(7,8,3,4,1,2,6,5)$. Since $8+1=9$ and $3+6=9$, we have $9\in\mathcal{L}_p(6)\cap\mathcal{R}_p(8)$. Since $7+8=15$, we have $15\in\mathcal{L}_p(8)$. Now $(15,9,6)$ is a monotone $2$-additive subsequence of $p$ which is a contradiction.
\end{proof}

\begin{lemma}\label{Lemma:Monotone3142}
Let $p\in S_{18}$. If the subpermutation of $p$ on $\{1,2,3,4\}$ is $(3,1,4,2)$, then $p$ has a monotone $2$-additive subsequence of length three.
\end{lemma}
\begin{proof}
    Let $p\in S_{18}$ be such that the subpermutation of $p$ on $\{1,2,3,4\}$ is $(3,1,4,2)$. Suppose, by way of contradiction, that $p$ does not have a monotone $2$-additive subsequence of length three. Since $2+4=6$, $6\in\mathcal{R}_p(4)$. Since $1+4=5$ and $3+2=5$, we have $5\in\mathcal{L}_p(4)\cap\mathcal{R}_p(3)$. Now we have $\{1,5\}\subseteq\mathcal{L}_p(6)$. Since $1+5=6$, we must have $5\in\mathcal{L}_p(1)$. Since $3+5=8$, we have $8\in\mathcal{L}_p(5)$. Now $\{2,6\}\subseteq\mathcal{R}_p(8)$. Since $2+6=8$, we must have $6\in\mathcal{R}_p(2)$. So the subpermutation of $p$ on $\{1,2,3,4,5,6\}$ is $(3,5,1,4,2,6)$.

    Since $3+4=7$ and $5+2=7$, we have $7\in\mathcal{L}_p(4)\cap\mathcal{R}_p(5)$. Now $\{1,7\}\subseteq\mathcal{R}_p(8)$. Since $1+7=8$, we have $7\in\mathcal{R}_p(1)$. Since $4+7=11$, we have $11\in\mathcal{R}_p(7)$. Since $3+7=10$ and $8+2=10$, we have $10\in\mathcal{L}_p(7)\cap\mathcal{R}_p(8)$. Since $7+6=13$, we have $13\in\mathcal{R}_p(7)$. Now we have $\{3,10\}\subseteq\mathcal{L}_p(13)$. Since $3+10=13$, we have $10\in\mathcal{L}_p(3)$. Since $5+7=12$ and $10+2=12$, we have $12\in\mathcal{L}_p(7)\cap\mathcal{R}_p(10)$. Now the subpermutation of $p$ on $\{6,8,10,12\}$ is $(8,10,12,6)$. Since $12+6=18$, we must have $18\in\mathcal{R}_p(12)$. Then $(8,10,18)$ is a monotone $2$-additive subsequence of $p$ which is a contradiction. 
\end{proof}

\begin{lemma}\label{Lemma:Monotone3124}
Let $p\in S_{14}$. If the subpermutation of $p$ on $\{1,2,3,4\}$ is $(3,1,2,4)$, then $p$ has a monotone $2$-additive subsequence of length three.
\end{lemma}
\begin{proof}
     Let $p\in S_{14}$ be such that the subpermutation of $p$ on $\{1,2,3,4\}$ is $(3,1,2,4)$. Suppose, by way of contradiction, that $p$ does not have a monotone $2$-additive subsequence of length three. Since $3+2=5$ and $1+4=5$, we have $5\in\mathcal{L}_p(4)\cap\mathcal{R}_p(3)$. So the subpermutation of $p$ on $\{1,2,3,4,5\}$ is $(3,5,1,2,4)$, $(3,1,5,2,4)$, or $(3,1,2,5,4)$.

     {\bf Case 1}: The subpermutation of $p$ on $\{1,2,3,4,5\}$ is $(3,5,1,2,4)$. Since $3+5=8$, we have $8\in\mathcal{L}_p(5)$. Since $5+1=6$, we have $6\in\mathcal{R}_p(5)$. Now we have $\{2,6\}\subseteq\mathcal{R}_p(8)$. Since $2+6=8$, we have $6\in\mathcal{R}_p(2)$. Since $2+4=6$, we have $6\in\mathcal{L}_p(4)$. Since $5+2=7$ and $1+6=7$, we have $7\in\mathcal{L}_p(6)\cap\mathcal{R}_p(5)$. Since $6+4=10$, we have $10\in\mathcal{R}_p(6)$. Now $(3,7,10)$ is a monotone $2$-additive subsequence of $p$ which is a contradiction.

     {\bf Case 2}: The subpermutation of $p$ on $\{1,2,3,4,5\}$ is $(3,1,5,2,4)$. Since $5+2=7$ and $4+3=7$, we have $7\in\mathcal{L}_p(4)\cap\mathcal{R}_p(5)$. Since $1+5=6$, we have $6\in\mathcal{L}_p(5)$. Since $5+4=9$, we have $9\in\mathcal{R}_p(5)$. Notice that now $\{3,6\}\subseteq\mathcal{L}_p(9)$. Since $3+6=9$, we must have $6\in\mathcal{L}_p(3)$. Now the subpermutation of $p$ on $\{1,2,3,4,5,6\}$ is $(6,3,1,5,2,4)$. Since $6+2=8$ and $5+3=8$, we have $8\in\mathcal{L}_p(5)\cap\mathcal{R}_p(6)$. Since $6+8=14$, we have $14\in\mathcal{L}_p(8)$. Since $8+2=10$, $10\in\mathcal{R}_p(8)$. Now we have $\{4,10\}\subseteq\mathcal{R}_p(14)$. Since $4+10=14$, we have $10\in\mathcal{R}_p(4)$. Now $(3,7,10)$ is a monotone $2$-additive subsequence of $p$ which is a contradiction.

     {\bf Case 3}: The subpermutation of $p$ on $\{1,2,3,4,5\}$ is $(3,1,2,5,4)$. Since $5+4=9$, we have $9\in\mathcal{R}_p(5)$. Since $1+5=6$, $6\in\mathcal{L}_p(5)$. Now $\{3,6\}\subseteq\mathcal{L}_p(9)$. Since $3+6=9$, we have $6\in\mathcal{L}_p(3)$. Notice that now the subpermutation of $p$ on $\{1,2,3,4,5,6\}$ is $(6,3,1,2,5,4)$. Since $6+2=8$ and $3+5=8$,we have $8\in\mathcal{L}_p(5)\cap\mathcal{R}_p(6)$. Since $6+1=7$ and $2+5=7$, we have $7\in\mathcal{L}_p(5)\cap\mathcal{R}_p(6)$. Since $6+7=13$, we have $13\in\mathcal{L}_p(7)$. Now $\{13,8\}\subseteq\mathcal{L}_p(5)$. Since $5+8=13$, we must have $13\in\mathcal{R}_p(8)$. Similarly, we have $\{1,8\}\subseteq\mathcal{L}_p(9)$ and since $1+8=9$, we have $8\in\mathcal{L}_p(1)$. Since $6+8=14$, we have $14\in\mathcal{L}_p(8)$. Now $\{1,13\}\subseteq\mathcal{R}_p(14)$. Since $1+13=14$, we have $13\in\mathcal{R}_p(1)$. Since $8+2=10$ and $7+3=10$, we have $10\in\mathcal{L}_p(7)\cap\mathcal{R}_p(8)$. Now $(14,10,4)$ is a monotone $2$-additive subsequence of $p$ which is a contradiction.
\end{proof}

\begin{theorem}\label{Theorem:Monotone2Additive}
    $18$ is the smallest $n$ such that every $p\in S_n$ has a monotone $2$-additive subsequence of length three; that is, $g(3,2)=18$.
\end{theorem}
\begin{proof}
$(8,10,12,14,3,5,16,1,7,11,4,17,13,2,9,6,15)\in S_{17}$ does not have monotone 2-additive subsequences of length 3. Hence $g(3,2)\geq18$.

To show that $g(3,2)=18$, it suffices to show that every $p\in S_{18}$ has a monotone $2$-additive subsequence of length three. Suppose, by way of contradiction, that $p\in S_{18}$ does not have a monotone $2$-additive subsequence of length three. WLOG, we assume that the subpermutation of $p$ on $\{1,2\}$ is $(1,2)$. Since $1+2=3$, the subpermutation of $p$ on $\{1,2,3\}$ is either $(3,1,2)$ or $(1,3,2)$. By Lemma~\ref{Lemma:Monotone132}, the subpermutation of $p$ on $\{1,2,3\}$ is $(3,1,2)$. Since $3+1=4$, we have $4\in\mathcal{R}_p(3)$. Hence, the subpermutation of $p$ on $\{1,2,3,4\}$ is $(3,4,1,2)$, $(3,1,4,2)$, or $(3,1,2,4)$. By Lemmas~\ref{Lemma:Monotone3412}, \ref{Lemma:Monotone3142}, and  \ref{Lemma:Monotone3124}, none of them is possible. Hence, we have a contradiction. Therefore, every $p\in S_{18}$ has a monotone $2$-additive subsequence of length three.
\end{proof}

Next, we look at the number of monotone 2-additive subsequences of length three contained in permutations. For positive integers $k$, $\ell$, $n$ and $p\in S_n$, let $G_p(k,\ell,n)$ be the number of monotone $\ell$-additive subsequences of length $k$ in $p$ and let 
\[
G(k,\ell,n)=\min_{p\in S_n}G_p(k,\ell,n).
\]
\begin{theorem}
For all large enough $n$, we have
    \[
    \frac{n}{18(\log n-\log18+2)}-6\leq G(3,2,n)\leq\frac{1}{18}n^2+\frac{7}{6}n.
    \]
\end{theorem}
\begin{proof}
The proof of the lower bound is similar to the proof of the lower bound in Theorem~\ref{Theorem:NumberOfAdditiveGeneral}. Consider large enough $n$ and $p\in S_n$. Then there are at least $(n/18)/[\log(n/18)+2]=n/[18(\log n-\log18+2)]$ primes $a$ in $\{1,2,\ldots,n\}$ such that $18a\leq n$. Since there are $7$ primes in $\{1,2,\ldots,18\}$, there are at least $
n/[18(\log n-\log18+2)]-7
$
primes $a$ in $\{19,20,\ldots,n\}$ such that $18a\leq n$. For each prime $a$, the subpermutation of $p$ on $\{a,2a,\ldots,18a\}$ has a monotone $2$-additive subsequence of length three. Since each of these monotone $2$-additive subsequence has a unique prime factor greater than $18$, these subsequences are all distinct. Adding a monotone $2$-additive subsequence in the supermutation of $p$ on $\{1,2,\ldots,18\}$, $p$ has \textit{at least} $
n/[18(\log n-\log18+2)]-6
$ monotone $2$-additive subsequences of length three. This completes the proof of the lower bound.

Now we prove the upper bound. Note that if we follow the same argument as in the proof of the upper bound in Theorem~\ref{Theorem:NumberOfAdditiveGeneral}, we will get a quadratic upper bound with a larger coefficient for the quadratic term. To get the upper bound stated in the current theorem, instead we will count the number of monotone $2$-additive subsequences for a specific family of permutations. Let $n$ be sufficiently large. For $2\leq a\leq\lceil n/2\rceil$, we consider the permutation
    \[
    \left(\left\lceil\frac{n}{a}\right\rceil,\left\lceil\frac{n}{a}\right\rceil-1,\ldots,1,n,n-1,\ldots, \left\lceil\frac{n}{a}\right\rceil+1\right)\in S_n.
    \]
    We note that the above permutation was inspired by the work of Myers \cite{Myers2002} who studied the minimum number of monotone subsequences of permutations without arithmetic properties. This permutation does not contain increasing subsequences of length three. All the decreasing subsequences come from either the \textbf{first part} $\left(\left\lceil\frac{n}{a}\right\rceil,\left\lceil\frac{n}{a}\right\rceil-1,\ldots,1\right)$ or the \textbf{second part} $\left(n,n-1,\ldots, \left\lceil\frac{n}{a}\right\rceil+1\right)$. Since both parts are decreasing, the number of monotone $2$-additive subsequences for each part is equal to the number of solutions to $x+y=z$ with $x<y$ in the corresponding domain. Hence, the number of monotone $2$-additive subsequences in the first part is \textit{at most}
    \[
    \sum_{i=1}^{\lfloor n/(2a)\rfloor}\left(\left\lceil\frac{n}{a}\right\rceil-2i\right)\leq\sum_{i=1}^{\lfloor n/(2a)\rfloor}\left(\frac{n}{a}+1-2i\right)\leq\left(\frac{n}{a}+1\right)\frac{n}{2a}-\left(\frac{n}{2a}-1\right)\frac{n}{2a}=\frac{1}{4a^2}n^2+\frac{1}{a}n;
    \]
    and the number of decreasing $2$-additive subsequences in the second part is \textit{at most}
    \[
    \begin{split}
    \sum_{i=\left\lceil n/a\right\rceil+1}^{\lfloor n/2\rfloor}(n-2i)=&n\left(\left\lfloor\frac{n}{2}\right\rfloor-\left\lceil\frac{n}{a}\right\rceil\right)-\left[\left\lfloor\frac{n}{2}\right\rfloor\left(\left\lfloor\frac{n}{2}\right\rfloor+1\right)-\left\lceil\frac{n}{a}\right\rceil\left(\left\lceil\frac{n}{a}\right\rceil+1\right)\right]\\\leq&n\left(\frac{n}{2}-\frac{n}{a}\right)-\frac{n}{2}\left(\frac{n}{2}-1\right)+\frac{n}{a}\left(\frac{n}{a}+2\right)=\frac{a^2+4-4a}{4a^2}n^2+\frac{a+2}{2a}n.
    \end{split}
    \]
    Adding them together, the total number of decreasing $2$-additive subsequences is \textit{at most}
    \[
    \frac{a^2+5-4a}{4a^2}n^2+\frac{a+4}{2a}n.
    \]
    For all large enough $n$, the above quantity reaches minimum $n^2/18+7n/6$ when $a=3$. This completes the proof of the upper bound.
\end{proof}

\section{Multiplicative and Inverse-Additive Subsequences}\label{Section:Prod}
In this section, we show that Theorems~\ref{Theorem:Main} and \ref{Theorem:Mon} can be naturally extended to subsequence products and inverse sums of subsequences. We call a (monotone) sequence $(a_1,a_2,\ldots,a_k)$ of positive integers (monotone) $\ell$-\textbf{multiplicative} if $a_1a_2\cdots a_k=a_1^\ell$ or $a_k^\ell$; and we call a (monotone) sequence $(a_1,a_2,\ldots,a_k)$ of positive integers (monotone) $\ell$-\textbf{inverse-additive} if $1/a_1+1/a_2+\cdots+1/a_k=\ell/a_1$ or $\ell/a_k$.

Our first result says that the existence of additive subsequences guarantees the existence of multiplicative subsequences in longer permutations. To show this, we use a well-known combinatorial argument in arithmetic Ramsey theory which essentially says that if an additive property is preserved under colorings of $\{1,2,\ldots,n\}$ then a similar multiplicative property is preserved under colorings of $\{2^1,2^2,\ldots,2^n\}$ (see, for example, \cite{ACOS2025}).

\begin{theorem}\label{Theorem:MultiExistence}
Let $k\geq3$, $\ell\geq2$, and $n$ be positive integers. If every $p\in S_n$ has an $\ell$-additive subsequence of length $k$, then every $q\in S_{2^n}$ has an $\ell$-multiplicative subsequence of length $k$. Similarly, if every $p\in S_n$ has a monotone $\ell$-additive subsequence of length $k$, then every $q\in S_{2^n}$ has a monotone $\ell$-multiplicative subsequence of length $k$.
\end{theorem}
\begin{proof}
We will prove the case for $\ell$-additive subsequences and the case for monotone $\ell$-additive subsequences is similar.

    Suppose every $p\in S_n$ has an $\ell$-additive subsequence of length $k$. Let $q\in S_{2^n}$. Consider the subpermutation $q'=(q'_1,q'_2,\ldots,q'_n)$ of $q$ on $\{2^1, 2^2,\ldots,2^n\}$. Let $p=(p_1,p_2,\ldots,p_n)\in S_n$ such that $q_i' = 2^{p_i}$ for all $i\in\{1,2,\ldots,n\}$. Since $p\in S_n$, $p$ has an $\ell$-additive subsequence of length $k$, say $(x_1,x_2,\ldots,x_k)$. WLOG, we assume that $x_1+x_2+\cdots+x_k=\ell x_k$. By the construction of $p$, $(2^{x_1},2^{x_2},\ldots,2^{x_k})$ is a subsequence of $q'$ and hence a subsequence of $q$. Since
    \[
    \prod_{i=1}^k2^{x_i}=2^{\sum_{i=1}^kx_i}=2^{\ell x_k}=\left(2^{x_k}\right)^\ell,
    \]
    $(2^{x_1},2^{x_2},\ldots,2^{x_k})$ is an $\ell$-multiplicative subsequence of $q$.
\end{proof}

Theorems~\ref{Theorem:Main} and \ref{Theorem:MultiExistence} together imply the existence of $2$-multiplicative subsequences of arbitrary length.

\begin{corollary}
    For any $k\geq3$, if $n$ is large enough, then every $p\in S_n$ has a $2$-multiplicative subsequence of length $k$.
\end{corollary}

Similarly, Theorems~\ref{Theorem:Mon} and \ref{Theorem:MultiExistence} together imply the existence of monotone $2$-multiplicative subsequences of length three.

\begin{corollary}
    If $n$ is large enough, then every $p\in S_n$ has a monotone $2$-multiplicative subsequence of length three.
\end{corollary}

Next, we prove that the existence of additive subsequences implies the existence of inverse-additive subsequences. Similar to Theorem~\ref{Theorem:MultiExistence}, we use another well-known argument in arithmetic Ramsey theory which essentially says that if an additive property is preserved under colorings of $\{1,2,\ldots,n\}$ then an ``inverse-additive" property is preserved under colorings of $\{1,2,\ldots,L_n\}$ where $L_n$ is the least common multiple of $\{1,2,\ldots,n\}$ (see, for example, \cite{BrownRodl1991,Gaiser2024,Lefmann1991}).
\begin{theorem}\label{Theorem:InverseExistence}
Let $k\geq3$, $\ell\geq2$, and $n$ be positive integers and let $L_n$ be the least common multiple of $\{1,2,\ldots,n\}$. If every $p\in S_n$ has an $\ell$-additive subsequence of length $k$, then every $q\in S_{L_n}$ has an $\ell$-inverse-additive subsequence of length $k$. Similarly, if every $p\in S_n$ has a monotone $\ell$-additive subsequence of length $k$, then every $q\in S_{L_n}$ has a monotone $\ell$-inverse-additive subsequence of length $k$.
\end{theorem}
\begin{proof}
We prove the case for inverse-additive subsequences. The case for monotone inverse-additive subsequences can be proved similarly.

Suppose every $p\in S_n$ has an $\ell$-additive subsequence of length $k$. Let $q\in S_{L_n}$. Let $q'$ be the subpermutation of $q$ on $\{L_n/1,L_n/2,\ldots,L_n/n\}$. Let $p\in S_n$ be such that for all $a,b\in\{1,2,\ldots,n\}$, $a\in\mathcal{L}_{p}(b)$ if and only if $L_n/a\in\mathcal{L}_{q'}(L_n/b)$. Let $(x_1,x_2,\ldots,x_k)$ be an $\ell$-additive subsequence of $p$. WLOG, we assume that $x_1+x_2+\cdots+x_k=\ell x_k$. Then $(L_n/x_1,L_n/x_2,\ldots,L_n/x_k)$ is a subsequence of $q'$ and
\[
\frac{1}{L_n/x_1}+\frac{1}{L_n/x_2}+\cdots+\frac{1}{L_n/x_k}=\frac{\ell}{L_n/x_k}.
\]
Hence, $(L_n/x_1,L_n/x_2,\ldots,L_n/x_k)$ is an $\ell$-inverse-additive subsequence of $q$.
\end{proof}

Theorems~\ref{Theorem:Main} and \ref{Theorem:InverseExistence} together imply the existence of $2$-inverse-additive subsequences of any length.
\begin{corollary}
    For all positive integers $k\geq3$, if $n$ is large enough, then every $p\in S_n$ has a $2$-inverse-additive subsequence of length $k$.
\end{corollary}
Similarly, Theorems~\ref{Theorem:Mon} and \ref{Theorem:InverseExistence} together imply the existence of monotone $2$-inverse-additive subsequences of length three.
\begin{corollary}
    If $n$ is large enough, then every $p\in S_n$ has a monotone $2$-inverse-additive subsequence of length three.
\end{corollary}
\section{Concluding Remarks}\label{Section:Conclu}
In this paper, we proved the existence of several subsequences with certain arithmetic properties in large enough permutations. Any answer to Question~\ref{Problem:l-Additive} and/or Question~\ref{Problem:Mon} would provide more insights on the behavior of sums of subsequences in permutations. Another possible direction is to study weighted subsequence sums; that is, $c_1x_1+c_2x_2+\cdots+c_kx_k$ where $(x_1,x_2,\ldots,x_k)$ is a sequence of length $k$ and $c_1,c_2,\ldots,c_k$ are constants.

In Theorems~\ref{Theorem:Existence:Nonmonotone} and \ref{Theorem:NonmonotoneLower}, we provided polynomial bounds for the smallest $n$ such that every $p\in S_n$ has a $2$-additive subsequence of length $k$. When $k=3$, we showed that the smallest $n$ is $5$ in Proposition~\ref{Prop:Extremal2Add}. This fact and Theorem~\ref{Theorem:MultiExistence} together imply that the smallest $n$ such that every $p\in S_n$ has a $2$-multiplicative subsequence of length three is at most $2^5=32$. In fact, 32 is the smallest $n$ such that every $p\in S_n$ has a $2$-multiplicative subsequence of length three because 
\[
(1,5,10,20,7,14,28,2,26,13,22,11,21,30,16,15,8,4,24,12,6,18,3,27,9,17,19,23,25,29,31)
\]
is a permutation of $\{1,2,\ldots,31\}$ which does not have $2$-multiplicative subsequences of length three. We have a similar result for $2$-inverse-additive subsequences of length three. Proposition~\ref{Prop:Extremal2Add} and Theorem~\ref{Theorem:InverseExistence} together imply that the smallest $n$ such that every $p\in S_n$ has a $2$-inverse-additive subsequence of length three is at most $60$ ($60$ is the least common multiple of $\{1,2,3,4,5\}$); and, on the other hand, if $p\in S_{59}$ such that the subpermutation of $p$ on \[\{2,3,4,5,6,7,8,9,10,12,14,15,16,18,20,21,24,27,28,30,35,36,40,42,45,48,54,56\}\]
    is
    \[
    (27,21,14,56,45,24,18,42,15,6,10,2,7,8,40,5,4,20,3,16,12,30,9,36,48,28,35,54),
    \]
then $p$ does not have $2$-inverse-additive subsequences of length three. It would be interesting to study whether, as $k$ approaches infinity, these general thresholds of $n$ for multiplicative subsequences and inverse-additive subsequences are significantly smaller than the upper bounds suggested by Theorems~\ref{Theorem:MultiExistence} and \ref{Theorem:InverseExistence}. The work of Brown and R\"{o}dl \cite{BrownRodl1991} and the first author \cite{Gaiser2024} provided affirmative answers to the counterpart of inverse-additive subsequences in arithmetic Ramsey theory; while it is known that the counterpart of multiplicative subsequences in arithmetic Ramsey theory behaves differently \cite[Section 6]{ACOS2025}.

\section*{Acknowledgments}
Some of the results in this article were previously included in the doctoral dissertation of the first author \cite{Gaiser2024b}. The authors used SageMath for computational experiments and counterexample searches, and some of the SageMath code was developed with the help of Google AI tools.

\end{document}